\documentclass[oneside,english]{amsart}

\usepackage[T1]{fontenc}
\usepackage[latin9]{inputenc}
\usepackage{amsthm}
\usepackage{amsbsy}
\usepackage{amstext}
\usepackage{amssymb}
\usepackage{graphicx}
\PassOptionsToPackage{normalem}{ulem}
\usepackage{ulem}
\usepackage{amsfonts}

\makeatletter
\numberwithin{equation}{section}
\numberwithin{figure}{section}
\theoremstyle{plain}

  \theoremstyle{definition}
  
  \theoremstyle{plain}
\newtheorem{lem}{Lemma}

\DeclareRobustCommand*\cal{\@fontswitch\relax\mathcal}

\makeatother

\usepackage{babel}
  \providecommand{\definitionname}{Definition}
  
\providecommand{\theoremname}{Theorem}

\begin{document}

\title{Pattern Recognition in a Ring of Delayed Phase Oscillators}

\author{Jan~Philipp~Pade,
        Serhiy~Yanchuk
        and~Liang~Zhao}
\thanks{This work was financially supported by the DFG in the framework of the International
Research Training Group (IRTG) 1740.}% <-this % stops a space
\thanks{J. P. Pade and S. Yanchuk are with the Humboldt-University of Berlin, Institute of
Mathematics, Unter den Linden 6, D-10099 Berlin, Germany.}% <-this % stops a space
\thanks{Z. Liang is with the University of Sao Paulo, Department of Computing and Mathematics,
Av. Bandeirantes, 3900, DCM-FFCLRP-USP, 14040-901, Ribeirao Preto - SP, Brazil.}

\begin{abstract}
We show that a ring of phase oscillators coupled with transmission delays
can be used as a pattern recognition system.
The introduced model encodes patterns as stable periodic orbits.
We present a detailed analysis of the underlying dynamics. In particular, 
we show that the system possesses a multitude of periodic solutions, prove stability
results and present a bifurcation analysis. Furthermore, we show successful
recognition results using artificial patterns and speech recordings.
\end{abstract}

\maketitle

\section{Introduction}
 
The ability to recognize encoded patterns is a very general and omnipresent
skill in humans and animals. Here, the term pattern can refer to a
structure of any nature activating a sensory perception. However,
recognition of visual and sonic patterns are the most common applications.
Developing artificial pattern recognition systems as well as investigating actual natural mechanisms of pattern recognition
in humans and animals has been a scientific challenge for decades
\cite{Malsburg1981,Grosberg1987,Kohonen2001,Wang2005a,Silva2012,Liang2013a}.
The most famous dynamical model is probably the one developed by Hopfield
in 1982 \cite{Hopfield1982}. It is a recurrent artificial neural
network consisting of coupled discrete elements. In contrast to feedforward
networks the underlying graph might have loops leading to more complex
dynamics and multistability, which in turn is used to store several
patterns in the network. Recently, there has been much effort in developing
more realistic models of human pattern recognition \cite{Borisyuk2013,Yamaguchi2003,Hopfield2009}.
It is evident that the pattern recognition task is closely related
to the memory function, and there is a strong evidence that memory
and memory recall are related to synchronization of neuronal ensembles
\cite{Jermakowicz2007}. On the other hand, in absence of stimuli
parts of the brain are known to be in a desynchronized state \cite{Deco2011}.
Consequently, models under consideration mostly consist of coupled
nonlinear dynamical systems that may possess complex dynamics \cite{Breve2009,Zhao2006}.

Delay coupled networks is another class of systems which attracted
much attention recently. Even a system consisting of a single dynamical
node with delayed feedback can be used effectively for information
processing purposes as it is evidenced in \cite{Appeltant2011,Larger2012}.
In \cite{Izhikevich2006} it is shown that a spiking network with
delayed connections can exhibit reproducible time-locked but not synchronous
firing patterns and the number of such patterns exceeds the number
of neurons in the network by far. Indeed, there is an increasing evidence
that the brain's memory and retrieval functions are closely connected
to spike timing \cite{Hasselmo2010}. 

Here, we present an evidence that pattern recognition tasks can be
performed by a system consisting of a single loop of delay coupled
phase oscillators. In contrast to the Hopfield model, patterns in
this model are not stored as fixed points but as periodic orbits.
In this way, time dimension is introduced in the information processing
device, which is biologically plausible. We also mention that the
simple ring structure makes it accessible to analytic investigations
and, at the same time, it serves as a model for central pattern generators
\cite{Golubitsky1999,strelkowa2011,Takamatsu2001,vishwanathan2011}.
The mechanism that we employ here is based on general properties of
rings of time-delayed systems with different connection delays, described
in \cite{Popovych2011,Yanchuk2011}. The delay transformation introduced
therein is used for analyzing the appearance of different patterns
in rings (as well as more complex networks \cite{Luecken2013b}) with
nonidentical delays. 

The article is structured as follows. In section \ref{sec:Synchronous-Motions}
we analyze the model equations with homogeneous delay. In particular,
we study coexisting periodic solutions in such a system and their
number and stability. Section \ref{sec:The-PR-Mechanism} introduces
the pattern recognition mechanism. In section \ref{sec:Numerical-Results}
we apply the model to various signals ranging from simple artificial
ones to speech and present numerical results. Section \ref{sec:Discussion}
concludes with a discussion and an outlook on various open questions.

\section{Synchronous Motions In Rings Of Phase Oscillators\label{sec:Synchronous-Motions}}

In this section we introduce the dynamical system used as a pattern
recognition device and present some useful results concerning its
dynamics. To our knowledge these results are new.
We consider a delayed, unidirectional ring of phase oscillators

\begin{equation}
\dot{x}_{j}\left(t\right)=\omega+\kappa\sin\left(x_{j+1}\left(t-\tau\right)-x_{j}\left(t\right)\right)\label{eq:PhaseOscillCont}
\end{equation}
where $x_{j}\in 2\pi\mathbb{R}/\mathbb{Z}$ and all the oscillators have the same individual frequency
$\omega$. The coupling between the oscillators is of strength $\kappa>0$ and delayed
by a time delay $\tau>0$. Provided with a global coupling structure these equations
are well known as Kuramoto model \cite{Kuramoto1984}. A lot of research
has been done on the synchronization properties of Kuramoto oscillators
with instantaneous coupling and related bifurcation scenarios since
the 80s \cite{Acebron2005}. Similarly, delayed Kuramoto oscillators
are well investigated \cite{Yeung1999,Lee2009,Chen2008}, whereas
there are fewer results on phase oscillators with a local coupling
structure, neither with instantaneous \cite{Rogge2004,Tilles2011,Tilles2013} nor delayed
coupling \cite{Earl2003}.

As our aim is to construct/memorize and recognize patterns by means
of synchronous states, we first have to assure the existence of the
synchronous solutions. We remark that by introducing new variables
$y_{j}=x_{j}-\omega$ and rescaling the time $t\mapsto\kappa t$ system
(\ref{eq:PhaseOscillCont}) transforms to $\dot{y}_{j}\left(t\right)=\sin\left(y_{j+1}\left(t-\tilde{\tau}\right)-y_{j}\left(t\right)\right)$.
However, we will consider (\ref{eq:PhaseOscillCont}), since it does
not make the analysis more complicated and still keeps the coupling
strength and frequency explicitly.

\subsection{Existence of synchronous periodic solutions\label{sub:Existence-of-sync-sol}}

Synchronous periodic solutions of (\ref{eq:PhaseOscillCont}) are
of the form
\begin{equation}
x_{j}\left(t\right)=x_{s}(t)=\Omega t\label{eq:PerSol}
\end{equation}
for some mean frequency $\Omega\in\mathbb{R}$. Substituting this
solution in system (\ref{eq:PhaseOscillCont}) yields the transcendental
equation for $\Omega$
\begin{equation}
\Omega=\omega-\kappa\sin\left(\Omega\tau\right).\label{eq:meanFrequency}
\end{equation}
Now, introducing the variable $z=\Omega\tau$ equation (\ref{eq:meanFrequency})
transforms to
\begin{equation}
\frac{z}{\kappa\tau}=\frac{\omega}{\kappa}-\sin(z)\label{eq:MeanFreqTransf}
\end{equation}
That is, we look for intersections of a straight line through the
origin with slope $\frac{1}{\kappa\tau}$ and a shifted sine, see
Fig. \ref{fig:PerSol} (a).

\begin{figure}
\includegraphics[width=3.5in]{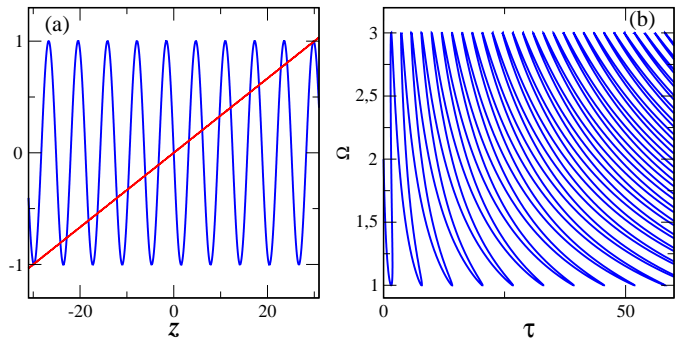}\caption{\label{fig:PerSol}
(a) Graphical representation of solutions to Eq.~(\ref{eq:MeanFreqTransf}).
The intersection points determine frequencies of synchronous solutions
$\Omega$, $z=\Omega\tau$. Parameters: $\omega=0$, $\tau=10$, and
$\kappa=3$. (b)~Frequency of synchronous solutions versus delay
$\tau$, given by Eq. (\ref{eq:Parametr_Freq}). Parameters: $\omega=2$
and $\kappa=1$.}
\end{figure}

In order to visualize how multiple synchronous solutions (\ref{eq:PerSol})
with different frequencies appear for different delay times $\tau$,
let us represent the solutions of (\ref{eq:meanFrequency}) in the
following parametric form
\begin{equation}
\begin{aligned}\tau\left(s\right) & =\frac{s}{\omega+\kappa\sin\left(-s\right)}\\
\Omega\left(s\right) & =\frac{s}{\tau\left(s\right)}
\end{aligned}
\label{eq:Parametr_Freq}
\end{equation}
So the same branch of periodic solutions reappears slightly inclined
for larger values of $\tau$, see Fig. \ref{fig:PerSol} (b). Both
figures show that indeed, either increasing $\tau$ or $\kappa$ leads
to an increase of solutions of the form (\ref{eq:PerSol}). Actually,
it is known that the number of coexisting periodic solutions in systems
with time delay $\tau$ grows at least linearly with $\tau$ \cite{Yanchuk2009}.
In our case, using simple geometrical arguments, one can show that
the number of such solutions is approximately $2\kappa\tau/\pi$.
A proof together with the stability analysis is given in the following
section.

\subsection{Stability of synchronous solutions}

The main result of this section is, that half of the above mentioned
$2\kappa\tau/\pi$ synchronous solutions are linearly stable, whereas
the other half is unstable. To determine the linear stability consider
the linearization of equation (\ref{eq:PhaseOscillCont}) along a
synchronous solution of the form (\ref{eq:PerSol})
\[
\dot{\xi}\left(t\right)=-K\xi\left(t\right)+KG\xi\left(t-\tau\right),
\]
where $K=\kappa\cos\left(\Omega\tau\right)$ and $G=\{g_{ij}\}$ is
the adjacency matrix of the unidirectional ring coupling, i.e. $g_{ij}=1$
if $j=i+1\,\mbox{mod}\, N$ and $g_{ij}=0$ otherwise. Here $N$ is
the number of oscillators in the ring. Thus the linear stability is
given by the distribution of zeros $\mu$ of the following characteristic
equation 
\begin{eqnarray*}
F\left(\mu\right) & = & \det\left(-\mu\mbox{Id}-K\mbox{Id}+Ke^{-\mu\tau}G\right)\\
 & = & -\left(\mu+K\right)^{N}+K^{N}e^{-N\mu\tau}.
\end{eqnarray*}
The solutions $\mu$ of this equation are the characteristic exponents
(eigenvalues) whose real parts determine the stability. The obtained
characteristic equation can be factorized to the set of $N$ more
simple equations:
\begin{equation}
0=\mu+K-Ke^{-\mu\tau}e_{n},\quad n=0,\dots,N-1,\label{eq:CharEq2}
\end{equation}
where $e_{n}=\exp\left[i2\pi n/N\right]$ are $N$-th roots of unity.
Now we can prove the following lemma.
\begin{lem}
Let $x\left(t\right)=\Omega t$ be a synchronous solution of (\ref{eq:PhaseOscillCont}).
The solution is linearly stable iff the following condition holds
\[
\cos\left(\Omega\tau\right)>0.
\]
Furthermore, if $\cos\left(\Omega\tau\right)<0$, the following upper
bound for the real parts of the characteristic exponents $\mu$ holds
true 
\[
\Re\left(\mu\right)\leq-2\kappa\cos\left(\Omega\tau\right).
\]
 \end{lem}
\begin{proof}
Splitting equation (\ref{eq:CharEq2}) into real and imaginary parts
($\mu=x+iy$) yields
\begin{eqnarray}
0 & = & x+K\left(1-e^{-x\tau}\cos\left(\frac{2\pi n}{N}-y\tau\right)\right),\label{eq:xeq}\\
0 & = & y-Ke^{-x\tau}\sin\left(\frac{2\pi n}{N}-y\tau\right).\nonumber 
\end{eqnarray}
Consider the two cases $K>0$ and $K<0$:\\
i) Case $K>0$. For $x>0$ we have $x+K\left(1-e^{-x\tau}\cos\left(\frac{2\pi n}{N}-y\tau\right)\right)>0$
which contradicts equation (\ref{eq:xeq}). Consequently, equation
(\ref{eq:xeq}) can only be fulfilled for $x<0$. It is shown in \cite{Earl2003}
that $K>0$ is also a necessary condition for stability which concludes the proof
of the first part of the lemma.\\
ii) Case $K<0$. If we have $x\leq0$ the upper bound is automatically
valid. So suppose we have a solution $\left(x,y\right)$ with $x>0$.
Then the equation (\ref{eq:xeq}) implies
\[
0<-K\left(1-e^{-x\tau}\cos\left(\frac{2\pi n}{N}-y\tau\right)\right)<-2K,
\]
so for $x>-2K$ we would have $x+K\left(1-e^{-x\tau}\cos\left(\frac{2\pi n}{N}-y\tau\right)\right)>0$
which means that indeed, the real part $x$ of $\mu$ can be bounded
as follows 
\[
x\leq-2K=-2\kappa\cos\left(\Omega\tau\right).
\]

\end{proof}
Using the condition for stability from lemma 1 we obtain estimations
for the number of coexistent stable and unstable periodic orbits.
\begin{lem}
Let $2\tau\kappa>\pi$. Then there are at least $\frac{\kappa\tau}{\pi}-\frac{1}{2}$
stable and unstable periodic synchronous solutions of equation (\ref{eq:PhaseOscillCont}),
respectively.\end{lem}
\begin{proof}
Finding periodic solutions is equivalent to finding zeros of the function
$f\left(\Omega\right)=\Omega-\omega+\kappa\sin\left(\Omega\tau\right)$.
We have seen in the last lemma, that a solution with frequency $\Omega$
is stable if $\cos\left(\Omega\tau\right)>0$. This is the case for
$\Omega\tau\in\left(-\frac{\pi}{2},\frac{\pi}{2}\right)+2\pi l$.
Motivated by this we define the disjoint intervals 
\begin{eqnarray*}
I_{l} & = & \left(-\frac{\pi}{2\tau},\frac{\pi}{2\tau}\right)+\frac{2\pi l}{\tau}\qquad l\in\mathbb{Z},\\
J_{l} & = & \left(\frac{\pi}{2\tau},\frac{3\pi}{2\tau}\right)+\frac{2\pi l}{\tau}\qquad l\in\mathbb{Z},
\end{eqnarray*}
for which $\cos\left(\Omega\tau\right)>0$ and $\cos\left(\Omega\tau\right)<0$,
respectively. We first remark that we have to fulfill $\Omega\in\left(\omega-\kappa,\omega+\kappa\right)$
(see equation (\ref{eq:meanFrequency})). A straightforward calculation
shows that the number $M$ of intervals $I_{l}$ lying in $\left(\omega-\kappa,\omega+\kappa\right)$
satisfies
\[
\frac{\kappa\tau}{\pi}-\frac{1}{2}\leq M\leq\frac{\kappa\tau}{\pi}+\frac{1}{2}.
\]
Now in $I_{l}$ we have
\[
f'\left(\Omega\right)=1+\kappa\tau\cos\left(\Omega\tau\right)>0
\]
and further $f\left(-\frac{\pi}{2\tau}+\frac{2\pi l}{\tau}\right)\leq0\leq f\left(\frac{\pi}{2\tau}+\frac{2\pi l}{\tau}\right)$
for all $I_{l}\subset\left(\omega-\kappa,\omega+\kappa\right)$. So
$f$ has exactly one zero in each $I_{l}$ by monotonicity. By the
same argument there has to be an odd number of zeros in each $J_{l}$.
By the shape of $f$ we conclude that it has one zero in each $J_{l}$
as well. So the total number of zeros $m$ in intervals $J_{l}$ has
the same bounds as $M$.
\end{proof}

\subsection{Bifurcations of synchronous solutions}

In subsection \ref{sub:Existence-of-sync-sol} we have seen that synchronous
solutions undergo bifurcations as either the delay time $\tau$ or
the coupling strength $\kappa$ is varied. In this section we give
some insight into these bifurcations. A solution (\ref{eq:PerSol})
can change its stability when one of its characteristic exponents
crosses the imaginary axis $\mu=iy$, $y\in\mathbb{R}$. Substituting
$\mu=iy$ into (\ref{eq:CharEq2}) and splitting the obtained equation
into real and imaginary parts yields
\begin{equation}
\begin{aligned}0 & =K\left(1-\cos\left(\frac{2\pi n}{N}-y\tau\right)\right),\\
0 & =y-K\sin\left(\frac{2\pi n}{N}-y\tau\right).
\end{aligned}
\label{eq:iy}
\end{equation}
The first equation in (\ref{eq:iy}) is fulfilled if either $K=0$
or $\frac{2\pi n}{N}-y\tau=2\pi l$ for some $l\in\mathbb{Z}$. Plugged
into the second equation both conditions yield $y=0$. Thus, Hopf
bifurcations with $y\ne0$ cannot occur in the model. The only possible
destabilization corresponds to eigenvalues crossing the imaginary
axis with vanishing imaginary part $y$. On the other hand, every
periodic solution has a trivial characteristic exponent $\mu=0$ corresponding
to perturbations in the direction of the periodic motion. So in order
to look for a change of stability we have to find non-simple zeros
of the characteristic function. Such non-simple zeros are possible
if $F'\left(0\right)=-NK^{N-1}\left(K\tau+1\right)=0$. This is fulfilled
if either one of the following conditions holds
\begin{equation}
\begin{aligned}K & =0,\\
\tau K & =-1.
\end{aligned}
\label{eq:perbif}
\end{equation}
The following lemma describes the type of bifurcations taking place
if the conditions (\ref{eq:perbif}) hold, see also Fig. \ref{fig:BifDiag}.
We remind that $K=\kappa\cos(\Omega\tau)$.
\begin{lem}
Let $x_{\tau}\left(t\right)=\Omega\left(\tau\right)t$ be a branch
of synchronous solutions of equation (\ref{eq:PhaseOscillCont}) with
$\kappa\neq0$. Varying $\tau$ it undergoes a transcritical bifurcation
at $K=0$. If $\omega\neq\pi l\kappa$ for any $l\in2\mathbb{Z}+1$ it undergoes
a fold bifurcation at $K=-1/\tau$ with $K$ defined as above.
\end{lem}
\begin{proof}
As it was mentioned above, under conditions (\ref{eq:perbif}) an
eigenvalue crosses the imaginary axis with vanishing imaginary part.
So the observed bifurcations can either be transcritical or fold.
In a fold bifurcation two new solutions emerge. This happens exactly
when the right hand side in (\ref{eq:meanFrequency}) crosses the
identity (see Fig.~\ref{fig:PerSol}(a)), which is equivalent to
\begin{eqnarray*}
1 & = & \frac{\partial}{\partial\Omega}\left[\omega-\kappa\sin\left(\Omega\tau\right)\right]\\
 & = & -\tau\kappa\cos\left(\Omega\tau\right)=-\tau K
\end{eqnarray*}
This condition also shows that at $K=0$ no other solution emerges,
therefore, the bifurcation at $K=0$ is transcritical. A straightforward
calculation shows that the eigenvalues cross the imaginary axis with
nonzero speed indeed.\\

\end{proof}
Fig.~\ref{fig:BifDiag} shows the obtained bifurcation diagram for
synchronous solutions, including the stability information. Here we
used that for $\tau=0$ we have $K=\kappa>0$, so the synchronous
state is stable.

\begin{figure}[h]
\includegraphics[width=3.5in]{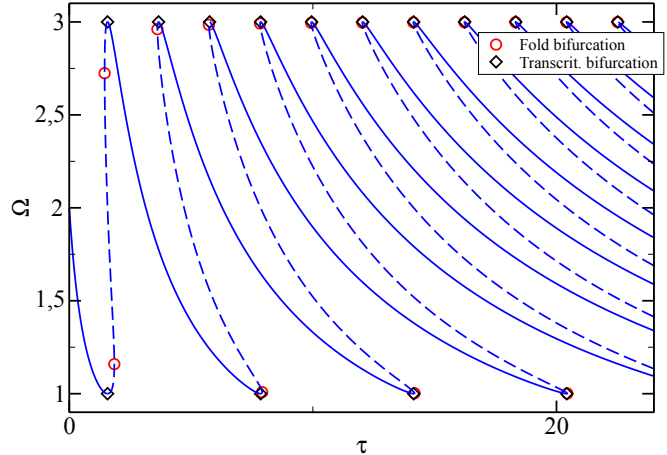}\caption{\label{fig:BifDiag}Bifurcation diagram for synchronous solutions
(\ref{eq:PerSol}). Frequency $\Omega$ is plotted versus time delay
$\tau$. Dotted lines represent unstable and plain lines stable solutions.
Parameters: $\omega=2$, $\kappa=1$. With increasing $\tau$, the
two bifurcations, transcritical and fold, are converging to each other. }
\end{figure}
We remark that, requiring $\sum_{j}a_{ij}=R\,,\,\forall i$ and appropriate
conditions on $f$, the above analysis can be done for the more general
system
\[
\dot{x}_{j}\left(t\right)=\omega+\frac{\kappa}{R}\sum_{j=1}^{N}a_{ij}f\left(x_{j}\left(t-\tau\right)-x_{i}\left(t\right)\right)
\]
with similar results. For the linear stability see \cite{Earl2003}.

\section{Pattern Recognition\label{sec:The-PR-Mechanism}}

In this section we present the idea of how pattern recognition can
be realized with the ring of delay coupled oscillators. We have seen
in the previous section that system (\ref{eq:PhaseOscillCont}) admits
coexisting synchronous solutions $x_{j}^{i}\left(t\right)=\Omega_{i}t$
with frequencies $\left\{ \Omega_{i}\right\} _{i\in I}\subset\left(\omega-\kappa,\omega+\kappa\right)$
given by equation (\ref{eq:meanFrequency}). In analogy with the firing
times of neuronal systems \cite{Goel2002,Canavier2010} we pay special
attention to the moments when an oscillator $x_{j}^{i}\left(t\right)$
reaches the boundary of the periodic domain $\left[0,2\pi\right]$,
i.e. $x_{j}\left(t_{j}\right)=2\pi$. This choice of phase is arbitrary,
and any other phase $\varphi_{*}$ instead of $2\pi$ can be chosen
instead \cite{Popovych(accepted)}. In what follows, we will use the
obtained crossing time sequences $t_{j}$ to represent patterns. In
section \ref{sub:Encoding} we explain how a given pattern can be
encoded in a ring of phase oscillators, see Fig.~\ref{fig:PR_Diagram}.
Afterwards, in section \ref{sub:Recognition} we explain the dynamic
recognition process.
\begin{figure}[h]
\includegraphics[width=3.5in]{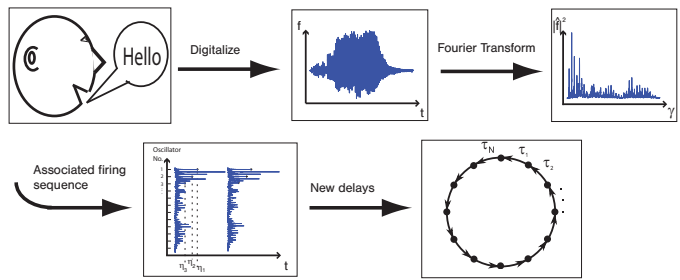}\caption{\label{fig:PR_Diagram}The process chart visualizes the encoding mechanism.
Step one is the recording and digitalization of a given stimulus.
The second step is not mandatory for pattern encoding. However, we
employ it in this article as it yields better results. It consists
of a fast Fourier transformation (FFT) of the recorded stimulus. Then,
a firing sequence $P$ is associated to the Fourier transform such
that each value of the FFT's square modulus is identified with the
firing time of an element of the ring. The last step is calculating
the coupling delays in the ring according to (\ref{eq:New_Delays}),
so the resulting dynamical system admits the firing sequence $P$
as a stable solution.}
\end{figure}

\subsection{Encoding\label{sub:Encoding}}

The idea for encoding a pattern in a ring of delay coupled oscillators
is based on a time shift transformation \cite{Yanchuk2011,Popovych2011}.
We shortly describe it here. We assume that a pattern $P$ is represented
by some $N$-dimensional vector $P=\left(p_{1},p_{2},...,p_{N}\right)\in\mathbb{R}^{N}$.
First, let us choose an $\Omega_{i}$ that corresponds to a stable
synchronous solution of (\ref{eq:PhaseOscillCont}). The timeshift
transformation 
\begin{equation}
y_{j}\left(t\right)=x_{j}\left(t-p_{j}\right)\label{eq:time_shift}
\end{equation}
converts (\ref{eq:PhaseOscillCont}) to the system 
\begin{equation}
\dot{y}_{j}\left(t\right)=\omega+\kappa\sin\left(y_{j+1}\left(t-\tau_{j}\right)-y_{j}\left(t\right)\right),\label{eq:Transformed_oscillators}
\end{equation}
which has the same form as (\ref{eq:PhaseOscillCont}) except that
the new delays 
\begin{equation}
\tau_{j}:=\tau-\left(p_{j+1}-p_{j}\right)\label{eq:New_Delays}
\end{equation}
 are non-identical. Here we assume that $\tau_{j}\ge0$, which is
always possible by taking $\tau$ sufficiently large. Under the transformation
(\ref{eq:time_shift}), the synchronous state $x_{j}\left(t\right)=\Omega_{i}t$
transforms into the solution
\[
y_{j}\left(t\right)=\Omega_{i}\cdot\left(t-p_{j}\right)
\]
of system (\ref{eq:Transformed_oscillators}). The stability properties
of this solution are also preserved \cite{Luecken2013b}. Thus, in
the system with non-homogeneous delays (\ref{eq:Transformed_oscillators}),
an oscillator $j$ crosses the phase $2\pi$ at times $p_{j},p_{j}+\frac{2\pi}{\Omega_{i}},...$.
Consequently, changing the delays $\tau\mapsto\tau_{j}$ yields the
new system system (\ref{eq:Transformed_oscillators}) that exhibits
exactly the required crossing sequence $P$ as a stable periodic solution.
In particular, the encoding is not a dynamic process but just consists
of the calculation of the $N$ delays $\tau_{j}$.

\subsection{Recognition\label{sub:Recognition}}

Now suppose $P$ is encoded as described above. Let 
\[
Q=P+\left(\delta_{1},\delta_{2},...,\delta_{N}\right)
\]
be a perturbed version of pattern $P$ that is to be recognized. We
associate an initial function for system (\ref{eq:Transformed_oscillators})
in a natural manner
\[
y^{0}\left(t\right)=\Omega_{i}\cdot\left(t-q_{1},t-q_{2},....,t-q_{N}\right)
\]
So for $\delta_{i}=0$ it corresponds to the synchronous state $\Omega_{i}t$
itself in transformed coordinates. Then, the recognition task consists
of deciding whether $Q$ is recognized as the previously encoded pattern
$P$. In order to do so we start system (\ref{eq:Transformed_oscillators})
with the initial function $y^{0}$ associated to $Q$ and stop after
a fixed time $T$. Then some measure of similarity between $\Phi_{T}\left(Q\right)$,
which is $Q$ evolved under the flow of (\ref{eq:Transformed_oscillators}),
and $P$ is considered to decide whether $Q$ should be recognized
as $P$. To specify the measure of similarity employed here, consider
the following straightforward calculation. Transformed back in original
coordinates the initial function $y^{0}$ reads 
\begin{eqnarray*}
x_{j}^{0}\left(t\right) & = & y_{j}^{0}\left(t+p_{j}\right)\\
 & = & \Omega_{i}\left(t+p_{j}-q_{j}\right)\\
 & = & \Omega_{i}\left(t-\delta_{j}\right).
\end{eqnarray*}
This means that solving system (\ref{eq:Transformed_oscillators})
with initial function $y^{0}$ (corresponding to $Q$) is equivalent
to solving the original system (\ref{eq:PhaseOscillCont}) with initial
function $x_{j}^{0}\left(t\right)$ (corresponding to $\Delta=\left(\delta_{1},\delta_{2},...,\delta_{N}\right)$).
In other words, pattern $Q$ is recognized as version of $P$ if the
orbit with initial function $x_{j}^{0}\left(t\right)$ corresponding
to $Q$ converges fast enough to the synchronous state with frequency
$\Omega_{i}$ (where {}``fast enough'' has to be quantified in a
specific situation). Consequently, we can use the order parameter
$r\left(t\right)=\sum_{j=1}^{N}\left|e^{i\Phi_{t}\left(x^{0}\right)}\right|$
as measure of similarity (reference and/or explain here?). It is worth
stressing the fact that the essential parameters for the pattern recognition
are the time $T$ after which we stop simulation of the differential
equations and a threshold for the order parameter that lets us decide
whether the flow after time $T$ is close enough to the synchronous
state, that is whether the pattern is recognized or not. The choice
of these parameters constitutes the learning phase.

\section{Numerical Results\label{sec:Numerical-Results}}

In this section we present numerical results for pattern recognition.
An important quantity for the mechanism is the size and shape of the
basin of attraction of the stable periodic solution representing the
encoded pattern, as the basin contains initial conditions (patterns)
that are possibly identified with the encoded pattern (see Fig. \ref{fig:Simple-pattern-perturbation2}
(b)). Obviously, we can not present a complete picture of the basins
of attraction as the phase space is infinite dimensional. However,
we will address some aspects of this question numerically in the following
sections.

\subsection{Recognition of artificial patterns}

We start with a simple pattern $P=\left(2,1,...,1\right)$. Fig. \ref{fig:Simple-pattern-perturbation2}
shows the order parameter after time $T=90$ for initial values $P+Q_{j}$,
where $Q_{j}=\left(0,...,0,\underbrace{\varepsilon}_{j},0,...,0\right)$.
Here $j$ determines the location and $\varepsilon$ the size of the
perturbation. We observe that, from a certain index $j$ on, initial
values $P+Q_{j}$ converge to the pattern with the same speed.

\begin{figure}[h]
\includegraphics[width=1.8in]{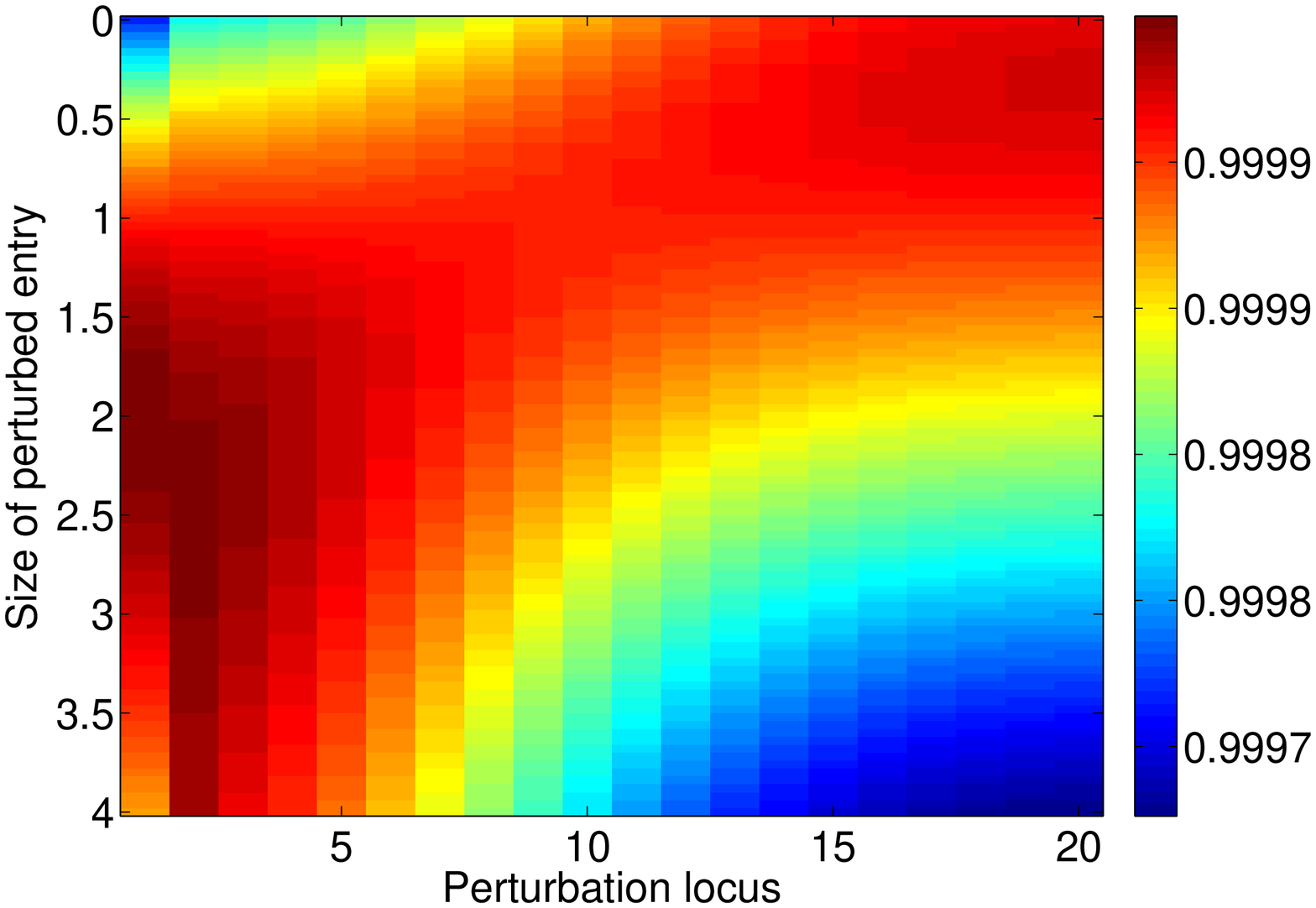}\includegraphics[width=1.8in]{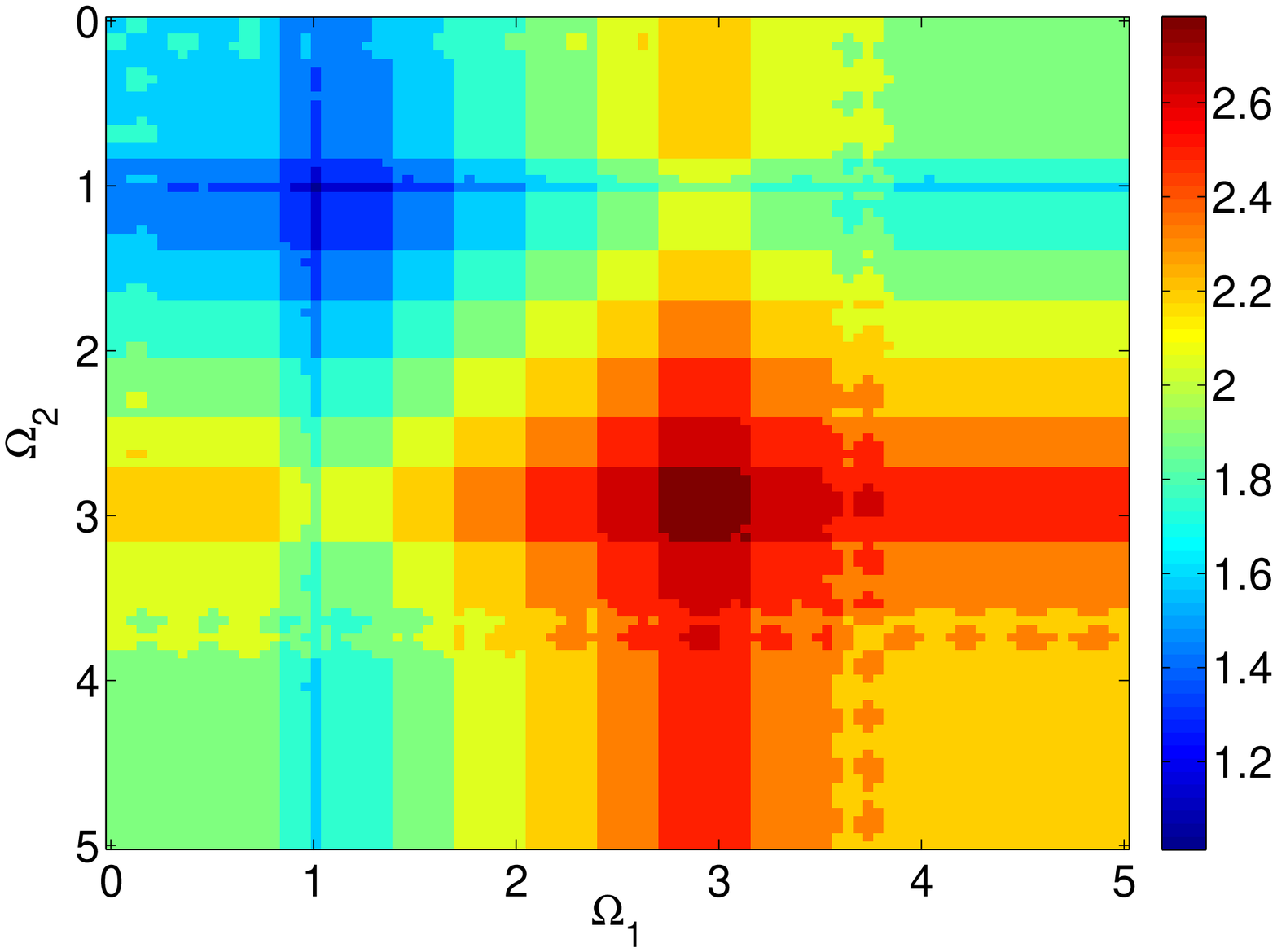}\caption{\label{fig:Simple-pattern-perturbation2}
(a) Encoded pattern $\left(2,1,...,1\right)$. Initial values are
$P+Q_{j}$ with $Q_{j}=\left(0,...,0,\protect\underbrace{\varepsilon}_{j},0,...,0\right)$,
where $j$ determines the location and $\varepsilon$ the size of
the perturbation. Color codes the order parameter measured after time
$T=900$. Parameters are $\tau=3,$ $\omega=1,$ $\kappa=1$. (b)
Visualization of attractor basins for two oscillators. The axis correspond
to initial values $\left[\Omega_{1}t,\Omega_{2}t\right]$. Color codes
the frequency of synchronous motion at time $T=50000$. Parameters
are $\omega=2$, $\kappa=1$ and $\tau=20$. For this value of $\tau$
we observe thirteen coexisting synchronous states of which seven are
stable which can be seen in Fig. \ref{fig:BifDiag}. }
\end{figure}

As a preparation for more complex audio signals we next investigate
recognition of simple sine waves. The task is the following: Suppose
two sine waves of different frequencies are encoded as two patterns
$P_{1}$ and $P_{2}$. It is to decide whether a given third wave
$Q$ is recognized as either one of the two encoded sines. Here, we
proceed as described in the flow chart Fig. \ref{fig:PR_Diagram}.
According to our scheme, the corresponding Fourier transforms $P_{1}$
and $P_{2}$ are the pattern associated with the signals. The use
of Fourier transform has the following advantages: one can cut off
the transformed signal above a certain frequency which can be seen
as a simple denoising process and at the same time it accelerates
the recognition process. Furthermore, for short enough time windows
the time course of the signal is not important and the Fourier transform
contains the main information independent of the signal's exact timing.
The pattern associated to a sine wave of a fixed frequency $f$ is
a delta peak $\delta_{f}$ (in an ideal case without noise). It is
evident that the error tolerance does not depend on the absolute value
of the frequency $f$ but rather on the frequency difference of the
perturbation and the pattern to be identified. 

Fig.~\ref{fig:Pattern_recog_sine} shows the results of two different
recognition processes. In the first recognition (a) two sine waves
of 5Hz and 15Hz are encoded. Again, the recognition process consists
of deciding whether a third given sine wave is recognized as either
one. It can be seen in the lower left plot that the 5Hz and the 15Hz
wave are recognized with an error tolerance of $\pm$0.1Hz. The same
results hold for recognition of the standard pitch A (440Hz) and one
semitone higher ($\sim$466Hz) (Figure (b)). In both cases the order
parameter is measured after time $T=50$ which corresponds to $\sim$17
times the delay time $\tau=3$. So the recognition is still relatively
fast. 

\begin{figure}[h]
\includegraphics[width=3.5in]{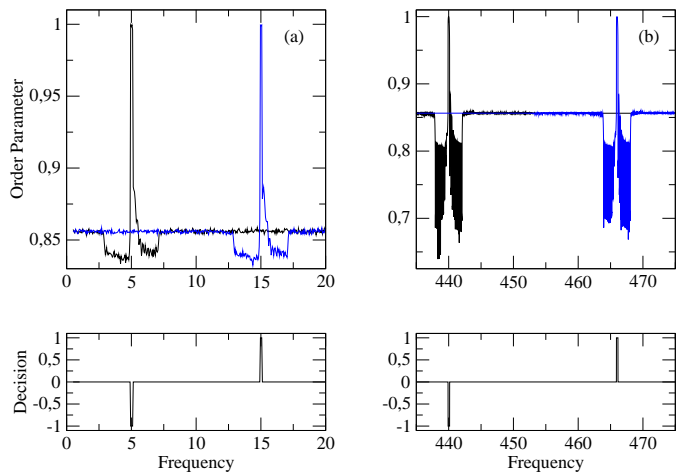}\caption{\label{fig:Pattern_recog_sine}Pattern recognition with two encoded
sine waves of frequencies 5Hz and 15Hz (a) and 440Hz and 466Hz (b).
A third wave which is to be recognized is used as initial value for
the system transformed according to the 5Hz and 15Hz patterns respectively
(see flow chart \ref{fig:PR_Diagram}). Upper plots show order parameters
after time $T=50$ for the different recognition processes. Parameters
are $\tau=3,\,\omega=1,\,\kappa=1$. Lower plots show decision of
the PR device. $-1$ stands for {}``Pattern 1 is chosen'', $1$
for {}``Pattern 2 is chosen'' and $0$ for none.}
\end{figure}

\subsection{Speech Recognition}

In this section, we use the model for speech recognition. As above
for sine waves, the pattern associated to an audio signal is its Fourier
transform. This is done to circumvent that otherwise, two identical
signals with different timing would not be recognized as the same.
With the transformed patterns we proceed in the same way as described
in section \ref{sec:The-PR-Mechanism}. The original audio signals
are the numbers from one to ten, spoken by two different voices (male
and female), each in three takes. So in total it is a set of 60 audio
tracks. First we encoded two randomly chosen different recordings.
For the recognition we randomly chose one of the four remaining recordings
that coincide with one of the two encoded numbers. In the previous
training phase we set the parameters to $\omega=1$, $\kappa=1$,
$\tau=3$, and $T=25$. The resulting recognition process was successful
with a rate of around 75\%.

\section{Conclusion\label{sec:Discussion}}

In this article, we have presented an evidence that a ring of delay
coupled oscillators can be used for information processing purposes.
In contrast to other models, the encoding of patterns is particularly
simple. For a signal of length $N$ it just consists of $N$ point
operations, the calculation of the new delays, in contrast to the
Hopfield model which requires the definition of
$N^{2}$ coupling weights. Furthermore, arbitrary patterns can be
encoded. In section \ref{sec:Synchronous-Motions} we presented an
extensive system's analysis. Using these results the model can be
controlled easily and parameters can be tuned suited to the application
on hand. Particularly, the number of stable periodic orbits can be
controlled and chosen arbitrarily high. Their frequencies, bounded
from below and above by $\omega-\kappa$ and $\omega+\kappa$ respectively,
can be easily determined numerically by solving the transcendent equation
\ref{eq:meanFrequency}. Results from section \ref{sec:Numerical-Results}
show that the model can be used as a reliable device for simple artificial
pattern recognition as well as for single word speech recognition. 

The aim of this article was to introduce main ideas, present some
analytical results for the model, as well as its potential capability
for information processing. Consequently, there are still open questions
and paths for future research. Here, we just mention a few. In section
\ref{sec:Synchronous-Motions} it was shown that for larger values
of the delay $\tau$ the model possesses a large number of coexisting
stable periodic orbits. This fact could possibly be exploited further.
Furthermore, considering Fourier transforms as patterns for speech
recognition was a probable choice, but possibly not the best. It is
an open question which preprocessing procedure suits best for the
model on hand. When classifying a pattern into two classes, we simply
measured the distance of the pattern to the two classes after a fixed
time by calculating the order parameter. This, of course, can be generalized
to more than two classes. Numerical results show that considering
the cross correlation instead of the order parameter yield similar
results. Yet, it is unclear which distance measure is best for a certain
class of signals. Furthermore, one could even think of coupling different
rings representing different classes in some way to yield a single
object representing a composite of classes. More generally the construction
above is still feasible for arbitrary coupling topologies. We only
presented results for artificial and audio signals. Encoding 2-D visual
patterns should be possible using the same ideas for a 2-D lattice.

\end{document}